\documentclass[11pt]{article}

\usepackage{amsmath}
\usepackage{amsthm}
\usepackage{amssymb}
\usepackage{bm}
\usepackage{mathtools}
\usepackage{mdframed}
\usepackage{microtype}
\usepackage{thmtools}
\usepackage{float}
\usepackage[obeyFinal,textsize=scriptsize,shadow]{todonotes}
\usepackage[Symbolsmallscale]{upgreek}
\usepackage{xcolor}
\usepackage{tikz}

\usepackage{custom}

\usepackage{comment,url,algorithm,algorithmic,graphicx,subcaption,relsize}
\usepackage{amssymb,amsfonts,amsmath,amsthm,amscd,dsfont,mathrsfs,mathtools,microtype,nicefrac,pifont}
\usepackage{float,psfrag,epsfig,color,url,hyperref}
\usepackage{upgreek}
\usepackage{epstopdf,bbm,mathtools,enumitem}
\usepackage[toc,page]{appendix}
\usepackage[mathscr]{euscript}
\usepackage[giveninits=true, style=alphabetic, maxbibnames=999, maxalphanames=999]{biblatex}

\usepackage[top=1in, bottom=1in, left=1in, right=1in]{geometry}

\def\balign#1\ealign{\begin{align}#1\end{align}}
\def\baligns#1\ealigns{\begin{align*}#1\end{align*}}
\def\balignat#1\ealign{\begin{alignat}#1\end{alignat}}
\def\balignats#1\ealigns{\begin{alignat*}#1\end{alignat*}}
\def\bitemize#1\eitemize{\begin{itemize}#1\end{itemize}}
\def\benumerate#1\eenumerate{\begin{enumerate}#1\end{enumerate}}

\newenvironment{talign*}
 {\csname align*\endcsname}
 {\endalign}
\newenvironment{talign}
 {\csname align\endcsname}
 {\endalign}

\def\balignst#1\ealignst{\begin{talign*}#1\end{talign*}}
\def\balignt#1\ealignt{\begin{talign}#1\end{talign}}

\let\originalleft\left
\let\originalright\right
\renewcommand{\left}{\mathopen{}\mathclose\bgroup\originalleft}
\renewcommand{\right}{\aftergroup\egroup\originalright}

\def\tinycitep*#1{{\tiny\citep*{#1}}}
\def\tinycitealt*#1{{\tiny\citealt*{#1}}}
\def\tinycite*#1{{\tiny\cite*{#1}}}
\def\smallcitep*#1{{\scriptsize\citep*{#1}}}
\def\smallcitealt*#1{{\scriptsize\citealt*{#1}}}
\def\smallcite*#1{{\scriptsize\cite*{#1}}}

\def\<{\left\langle} 
\def\>{\right\rangle}

\def\bigO#1{\mathcal{O}\left(#1\right)} 

\DeclareSymbolFont{rsfs}{U}{rsfs}{m}{n}
\DeclareSymbolFontAlphabet{\mathscrsfs}{rsfs}

\providecommand{\tr}{\mathop\mathrm{tr}}

\ifdefined\nonewproofenvironments\else
\ifdefined\ispres\else
\newtheorem{theorem}{Theorem}

\renewenvironment{proof}{\noindent\textbf{Proof.}\hspace*{.3em}}{\qed \vspace{.1in}}
\newenvironment{proof-sketch}{\noindent\textbf{Proof Sketch}
  \hspace*{1em}}{\qed\bigskip\\}
\newenvironment{proof-idea}{\noindent\textbf{Proof Idea}
  \hspace*{1em}}{\qed\bigskip\\}
\newenvironment{proof-of-lemma}[1][{}]{\noindent\textbf{Proof of Lemma {#1}}
  \hspace*{1em}}{\qed\\}
\newenvironment{proof-of-theorem}[1][{}]{\noindent\textbf{Proof of Theorem {#1}}
  \hspace*{1em}}{\qed\\}
\newenvironment{proof-attempt}{\noindent\textbf{Proof Attempt}
  \hspace*{1em}}{\qed\bigskip\\}

\newenvironment{remark}{\noindent\textbf{Remark.}
  \hspace*{0em}}{\smallskip}

\fi

\newtheorem{proposition}[theorem]{Proposition}

\theoremstyle{definition}

\fi
\numberwithin{equation}{section}
\makeatletter
\@addtoreset{equation}{section}
\makeatother

\hypersetup{
  colorlinks,
  linkcolor={red!50!black},
  citecolor={blue!50!black},
  urlcolor={blue!80!black}
}

\newcommand{\BR}{\mathbb R}

\usepackage{crossreftools}

\usepackage{tcolorbox}
\usepackage[normalem]{ulem}

\makeatletter
\renewcommand{\paragraph}{%
  \@startsection{paragraph}{4}%
  {\z@}{1.25ex \@plus 1ex \@minus .2ex}{-1em}%
  {\normalfont\normalsize\bfseries}%
}

\makeatother

\mathtoolsset{showonlyrefs}

\begin{document}

\title{Lower bounds for trace estimation via Block Krylov and other methods}

 \author{
  Shi Jie Yu\thanks{
  Department of Computer Science and Engineering at New York University, \texttt{sy4468@nyu.edu}.
 }
}

\pagenumbering{gobble}
\maketitle

\begin{abstract}
This paper studies theoretical lower bounds for estimating the trace of a matrix function, $\text{tr}(f(A))$, focusing on methods that use Hutchinson's method along with Block Krylov techniques. These methods work by approximating matrix-vector products like $f(A)V$ using a Block Krylov subspace. This is closely related to approximating functions with polynomials. We derive theoretical upper bounds on how many Krylov steps are needed for functions such as $A^{-1/2}$ and $A^{-1}$ by analyzing the upper bounds from the polynomial approximation of their scalar equivalent. In addition, we also develop lower limits on the number of queries needed for trace estimation, specifically for $\text{tr}(W^{-p})$ where $W$ is a Wishart matrix. Our study clarifies the connection between the number of steps in Block Krylov methods and the degree of the polynomial used for approximation. This links the total cost of trace estimation to basic limits in polynomial approximation and how much information is needed for the computation.
\end{abstract}

\newpage

\tableofcontents

\clearpage
\pagenumbering{arabic}

\section{Introduction}
Estimating the trace of matrix functions is a significant and widely applied task across diverse scientific and engineering fields, including statistical analysis, machine learning, and computational physics. Given its importance, understanding the fundamental limits of trace estimation is essential. Theoretical lower bounds play a critical role in this regard, as they define the best possible performance any algorithm can achieve. Our study investigates these performance limits. We utilize Block Krylov methods to analyze the capabilities of existing algorithms and employ Wishart matrices to construct and evaluate challenging instances for trace estimation problems. Establishing such theoretical lower bounds is inherently difficult because it requires demonstrating that no algorithm within a given computational model can surpass these identified limits. Nevertheless, deriving theoretical upper bounds can reveal the underlying complexities of the problem and provide a starting point for formulating tighter and more informative lower bounds.

\subsection{Our contributions}
This paper makes several contributions to understanding the theoretical bounds of trace estimation using Block Krylov and other methods:
\begin{itemize}
    \item \textbf{Formalizing Upper Bounds via Polynomial Approximation:} We provide a simple formalization on how the number of iterations in a Block Krylov method for approximating $f(A)V$ corresponds to the degree of a polynomial approximating the scalar function $f(t)$. Based on this formalization, we derive upper bounds on the computational complexity for the Hutchinson-Krylov trace estimation method.
    \item \textbf{Specific Upper Bounds for Key Functions:} We provide concrete upper bounds for estimating $\mathrm{tr}(A^{-1})$ as an extension to a previously established upper bound for estimating  $\mathrm{tr}(A^{-1/2})$.
    \item \textbf{Establishing Query Lower Bounds:} We present and extend query lower bounds for trace estimation, particularly for $\mathrm{tr}(W^{-p})$ where $W$ is a Wishart matrix. We show that any algorithm achieving a constant factor approximation with high probability requires at least $\Omega(d)$ matrix-vector product queries, where $d$ is the dimension of the matrix. This is achieved by leveraging properties of Wishart matrix posteriors under querying.
\end{itemize}

\subsection{Related work}\label{scn:related}
\subsubsection{Stochastic Trace Estimation}
A cornerstone in stochastic trace estimation is Hutchinson's method \cite{hutchinson1989stochastic}, which provides an unbiased estimator for the trace of a matrix, $\mathrm{tr}(M)$. This method utilizes quadratic forms $z^T Mz$ with random vectors $z$ and has been widely adopted and analyzed, as seen in works like \cite{avron2011randomized}. The practical application of Hutchinson's method to estimate $\mathrm{tr}(f(A))$ often involves the computation of $f(A)z_s$ for multiple random vectors $z_s$. Recent developments in this area have focused on enhancing these stochastic approaches, including randomized and sketching-based methods, sometimes combining them with Krylov techniques \cite{meyer2021hutch++}.

\subsubsection{Krylov Subspace Methods}
For large matrices $A$, the direct computation of $f(A)z_s$ required for trace estimation is often prohibitively expensive. Krylov subspace methods \cite{liesen2013krylov} offer an efficient alternative by approximating $f(A)z_s$ without needing to explicitly form the matrix $f(A)$. When approximating $f(A)V$ using $m$ steps of a Block Krylov method, the resulting iterate $F_m$ resides within the Block Krylov subspace $\mathcal{K}_m(A,V)$. This principle of leveraging Krylov subspaces to handle matrix functions is exploited in numerous algorithms.

\subsubsection{Polynomial Approximation}
The efficacy of Krylov subspace methods for approximating $f(A)V$ is deeply connected to polynomial approximation \cite{greenbaum1997iterative,frommer2008matrix}. The iterate $F_m$ obtained from a Block Krylov method can be expressed as $Q_{m-1}(A)V$, where $Q_{m-1}$ is a polynomial of degree $m-1$ that approximates the function $f$. The efficiency of such polynomial approximations has been studied extensively, especially for functions like $x^{-1/2}$ and $x^{-1}$ which are crucial for applications involving graph Laplacians and covariance matrices \cite{spielman2011spectral}. Chebyshev polynomials are often employed as they can provide near-optimal approximations on the interval $[-1,1]$, with these results being extendable to general intervals that contain the spectrum of $A$ \cite{rivlin2020chebyshev}. Key results concerning the degree of polynomials required to achieve a specific approximation accuracy are provided by \cite{sachdeva2014faster}, as referenced by Proposition \ref{prop:sacdeva} in this paper. These results directly inform the number of Krylov iterations necessary for an effective approximation.

\subsubsection{Lower Bounds and Information-Based Complexity}
Lower bounds in numerical computations establish fundamental limits on the performance of any algorithm. The field of information-based complexity investigates the minimum resources, such as matrix-vector products or queries to the matrix, required to solve a problem to a given accuracy \cite{traub2003information}. Such lower bounds are vital for assessing the optimality of algorithms. For instance, the work by Braverman et al. \cite{braverman2020gradient} on regression problems introduces techniques for deriving lower bounds using Wishart matrices, demonstrating how information about a matrix $W$ is progressively revealed through queries of the form $Wv_i$. Their Lemma 3.4, presented here as Proposition \ref{prop:wishart_posterior}, is crucial for understanding the structure of the posterior distribution of a Wishart matrix after several queries. The underlying properties of Wishart matrices, including their eigenvalue distributions as documented in works like \cite{edelman1988eigenvalues,vershynin2018high,szarek1991condition}, form the basis for these hardness results.

The interplay between the number of probe vectors $N_v$ in Hutchinson's method, the number of Krylov steps $D$, and the properties of the matrix $A$ and function $f$ continues to be an active area of research in achieving optimal trace estimators. Our work contributes to this by providing a focused analysis on the optimality of Block Krylov methods through the lens of polynomial approximation and query complexity.

\section{Preliminaries}

\subsection{Block Krylov as a polynomial approximation problem}
Block Krylov subspace methods are iterative algorithms \cite{saad2019iterative}  whose foundations are deeply intertwined with polynomial approximation theory. A Block Krylov subspace of order $m$, generated by a square matrix $A$ and a block of vectors $V$ (which itself is a matrix), is defined as:
$$ \mathcal{K}_m(A,V) = \mathrm{span}\{V, AV, A^2V, \ldots, A^{m-1}V\} $$
Any element $Y$ that belongs to this subspace $\mathcal{K}_m(A,V)$ can be written as a matrix polynomial in $A$ acting on $V$. That is, $Y=Q_{m-1}(A)V$, where $Q_{m-1}(t)$ is a polynomial of the form $c_0 + c_1 t + \ldots + c_{m-1}t^{m-1}$, having a degree of at most $m-1$.

When a Block Krylov method is run for $m$ iterations to approximate $f(A)V$, the resulting approximation, denoted $F_m$, is constructed within the Krylov subspace $\mathcal{K}_m(A,V)$. Consequently, $F_m$ can be expressed as $Q_{m-1}(A)V$ for some polynomial $Q_{m-1}$ of degree at most $m-1$. This polynomial $Q_{m-1}(t)$ is chosen, either implicitly or explicitly by the method, to serve as an approximation to the scalar function $f(t)$ on the spectrum of $A$. The degree of this ``approximating polynomial'' $Q_{m-1}(t)$ is therefore $m-1$.

In summary, the number of iterations $m$ in a Block Krylov method dictates the dimension of the Krylov subspace generated. This dimension, in turn, determines the maximum degree of the polynomials ($m-1$ for $Q_{m-1}$ or $m$ for $P_m$) that can be used to construct the approximation or manage the residual.

This link between iteration count $m$ and polynomial degree ($m-1$ for $Q_{m-1}$) is fundamental to establishing that the theoretical bounds of Block Krylov methods are equivalent to those of polynomial approximation methods. If polynomial approximation theory requires a degree $D$ polynomial to approximate a scalar function $f(x)$ to a certain precision (often determined by parameters like $\kappa$ and $\delta$), then a Block Krylov method for $f(A)V$ needs approximately $D+1$ iterations (i.e., $m \approx D+1$). Thus, the theoretical bound on the number of Krylov iterations $m$ directly scales with, and is of the same order as, the required polynomial degree $D$. For instance, if $D = \mathcal{O}(\sqrt{\kappa} \log(\kappa/\delta))$, the iteration count $m$ is also $\mathcal{O}(\sqrt{\kappa} \log(\kappa/\delta))$.

\subsection{Trace Estimation using Hutchinson's method and Block Krylov}

Estimating the trace of a matrix function, $\mathrm{tr}(f(A))$, for a large matrix $A$ and a general function $f$, is a significant problem in numerical linear algebra and scientific computing that is notoriously computationally expensive, especially as n grows, with direct methods often incurring $\mathcal{O}(n^3)$ costs. A prominent approach combines Hutchinson's stochastic trace estimator with polynomial approximations of $f(A)v$ typically computed via Block Krylov subspace methods.

\subsubsection{The Hutchinson-Krylov method}

Hutchinson's method approximates the trace of a matrix $M$ using the formula:
$$ \mathrm{tr}(M) = \mathbb{E}_{\mathbf{z}}[\mathbf{z}^T M \mathbf{z}] $$
where $\mathbf{z}$ is a random vector whose entries are typically drawn from a distribution with zero mean and unit variance (e.g., Rademacher or standard Gaussian). In practice, this expectation is estimated by averaging over $N_v$ random probe vectors $\{\mathbf{z}_s\}_{s=1}^{N_v}$:
$$ \mathrm{tr}(f(A)) \approx \frac{1}{N_v} \sum_{s=1}^{N_v} \mathbf{z}_s^T f(A) \mathbf{z}_s $$
For a general function $f$, the term $f(A)\mathbf{z}_s$ is not computed exactly. Instead, it is approximated by $Q_D(A)\mathbf{z}_s$, where $Q_D(t)$ is a polynomial of degree $D$ that approximates the scalar function $f(t)$ over an interval containing the spectrum of $A$. The matrix-vector product $Q_D(A)\mathbf{z}_s$ is efficiently computed using $D$ steps of a Krylov subspace method (such as Lanczos \cite{golub1977block} for symmetric $A$ or Arnoldi \cite{sadkane1993block} for non-symmetric $A$) initiated with $A$ and $\mathbf{z}_s$. This process effectively yields an approximation within the Krylov subspace $\mathcal{K}_D(A, \mathbf{z}_s)$.
Thus, the practical trace estimator becomes:
$$ \tilde{\mathrm{tr}}(f(A)) = \frac{1}{N_v} \sum_{s=1}^{N_v} \mathbf{z}_s^T Q_D(A) \mathbf{z}_s $$

\subsubsection{Reduction of trace estimation complexity to polynomial approximation bounds}\label{sec:trace_to_poly_reduction}

We can establish a direct relationship between the complexity of the underlying polynomial approximation and the overall computational complexity of the Hutchinson-Krylov trace estimation method.

Let us assume that, from polynomial approximation theory, the minimum degree $D$ of a polynomial $Q_D(t)$ required to approximate the scalar function $f(t)$ to a desired accuracy $\epsilon_{\text{poly}}$ (over a relevant domain related to $A$'s spectrum) is given by:
$$ D = \bigO(g(x)) $$
Here, $g(x)$ is a function representing this complexity, where $x$ encapsulates problem parameters such as the condition number $\kappa$ of $A$, the bounds of the approximation interval, and parameters related to the target accuracy $\epsilon_{\text{poly}}$

The computation of each term $Q_D(A)\mathbf{z}_s$ using $D$ steps of a Krylov method (or $D$ matrix-vector products if $Q_D(A)\mathbf{z}_s$ is formed by explicitly applying the polynomial) has a computational cost dominated by these matrix-vector products (MVPs). Thus, the cost for processing a single probe vector $\mathbf{z}_s$ is:
$$ \text{Cost per probe vector} = \bigO(D) \text{ MVPs} $$
Substituting the complexity for the polynomial degree $D$, we get:
$$ \text{Cost per probe vector} = \bigO(g(x)) \text{ MVPs} $$
Since Hutchinson's method requires averaging over $N_v$ such probe vectors, the total computational cost for estimating $\mathrm{tr}(f(A))$ is the product of the number of probe vectors and the cost per vector:
$$ \text{Total Cost} = N_v \times (\text{Cost per probe vector}) = N_v \times \bigO(g(x)) $$
Therefore, the overall complexity for the trace estimation is:
$$ \text{Total Cost} = \bigO(N_v \cdot g(x)) \text{ MVPs} $$
In conclusion, if the theoretical bound on the degree of the polynomial approximation required for $f(t)$ is $\bigO(g(x))$, then the computational complexity (in terms of dominant operations like MVPs) for estimating $\mathrm{tr}(f(A))$ using the Hutchinson-Krylov method scales as $\bigO(N_v \cdot g(x))$.

\section{Upper bound via Block Krylov}\label{sec:upperbk}
\begin{remark}
Generally, it is much harder to establish a theoretical lower bound for Block Krylov methods, or any stochastic methods for that matter, as compared to establishing an upper bound. Sometimes, establishing upper bounds can give us some insights on how the lower bounds may potentially manifest.
\end{remark}

\noindent
From Section~\ref{sec:trace_to_poly_reduction}, we can establish an upper bound on the computational complexity of estimating $\mathrm{tr}(f(A))$ using the Hutchinson-Krylov method by proxy, by first determining the upper bound for the associated polynomial approximation of the scalar function $f(x)$. As outlined, the Hutchinson-Krylov approach approximates $\mathrm{tr}(f(A))$ via terms of the form $\mathbf{z}_s^T Q_D(A) \mathbf{z}_s$, where $Q_D(A)\mathbf{z}_s$ is computed using $D$ steps of a Block Krylov method. The number of steps $D$ corresponds to the degree of the polynomial $Q_D(t)$ that approximates $f(t)$. If polynomial approximation theory guarantees that $f(x)$ can be approximated to the desired accuracy with a polynomial of degree $D = \mathcal{O}(g(x))$, then the total computational cost for the trace estimation, dominated by matrix-vector products, becomes $\mathcal{O}(N_v \cdot D) = \mathcal{O}(N_v \cdot g(x))$. Therefore, to establish the overall upper bound for the trace estimation via this Block Krylov-based framework, it suffices to prove the upper bound on the degree $D$ required for the polynomial approximation of $f(x)$.

\subsection{\texorpdfstring{Upper bound for $\tr(A^{-1/2})$: a technical review}{Upper bound for tr(A-1/2): a technical review}}
\begin{proposition}[{\cite[Theorem 3.3]{sachdeva2014faster}}] \label{prop:sacdeva}
    For any positive integer $s$ and $0 < \delta < 1,$ there exists a polynomial $p_{s, \delta}$ of degree $\lceil \sqrt{2s \ln (2/\delta)} \rceil$ such that $|p_{s, \delta}(x) - x^s| \le \delta$ for all $x \in [-1, 1]$.
\end{proposition}

\begin{proposition}[polynomial approximation of inverse square root] \label{prop:inverse_square_root_polyapprox}
    For any $\kappa \ge 2$ and $0 < \delta < \frac{1}{2}$, there exists a polynomial $q_{\kappa, \delta}$ of degree $O(\sqrt{\kappa} \log \frac{\kappa}{\delta})$ such that $|q_{\kappa, \delta}(x) - x^{-1/2}| \le \delta/\sqrt\kappa$ for all $1 \le x \le \kappa$.
\end{proposition}

\begin{proof}
    We begin by considering the Taylor series of $f(x) = (1+x)^{-1/2}$ around $x=0$:
    \[(1+x)^{-1/2} = \sum_{t=0}^\infty \binom{-1/2}{t} x^t = 1 + \sum_{t=1}^\infty c_t x^t\,,\]
    where $c_t = \binom{-1/2}{t}$ and $|c_t| \le 1$ for all $t \ge 1$.
    
    Truncating this series at $T$ terms, $P_T(x) = \sum_{t=0}^{T} c_t x^t$, the remainder $R_T(x) = (1+x)^{-1/2} - P_T(x)$ for $|x| \le 1 - 1/\kappa < 1$ satisfies $|R_T(x)| \le \kappa (1 - 1/\kappa)^{T+1}$. Choosing $T = O(\kappa \log \frac{\kappa}{\delta})$ ensures $|(1+x)^{-1/2} - \sum_{t=0}^{T} c_t x^t| \le \frac{\delta}{2}$ for $|x| \le 1 - \frac{1}{\kappa}$.

    Next, by Proposition~\ref{prop:sacdeva}, for each $t \in \{1, \dots, T\}$, there exists a polynomial $p_{t, \delta}(x)$ of degree $O(\sqrt{t \log (t/\delta)})$ such that $|p_{t, \delta}(x) - x^t| \le \frac{\delta}{4t^2}$ for all $x \in [-1, 1]$. Let $p_{0, \delta}(x) = 1$. Consider the polynomial $\hat{p}(x) = \sum_{t=0}^{T} c_t p_{t, \delta}(x)$. For $|x| \le 1 - \frac{1}{\kappa}$, the approximation error is:
    \begin{align*}
        \left|(1+x)^{-1/2} - \hat{p}(x)\right| &\le \left|(1+x)^{-1/2} - \sum_{t=0}^{T} c_t x^t\right| + \left|\sum_{t=0}^{T} c_t (x^t - p_{t, \delta}(x))\right| \\
        &\le \frac{\delta}{2} + \sum_{t=1}^{T} |c_t| |x^t - p_{t, \delta}(x)| \le \frac{\delta}{2} + \sum_{t=1}^{T} \frac{\delta}{4t^2} < \delta\,.
    \end{align*}
    The degree of $\hat{p}(x)$ is $O(\sqrt{T \log (T/\delta)}) = O(\sqrt{\kappa} \log \frac{\kappa}{\delta})$.

    Now, consider $y = x/\kappa - 1$. For $x \in [1, \kappa]$, $y \in [-1 + 1/\kappa, 0]$, which is within $[-1 + 1/\kappa, 1 - 1/\kappa]$. We have $|\hat{p}(y) - (1+y)^{-1/2}| \le \delta$. Substituting $y = x/\kappa - 1$:
    \[\left|\hat{p}\left(\frac{x}{\kappa} - 1\right) - \left(1 + \frac{x}{\kappa} - 1\right)^{-1/2}\right| \le \delta \implies \left|\hat{p}\left(\frac{x}{\kappa} - 1\right) - \sqrt{\frac{\kappa}{x}}\right| \le \delta\,.\]
    Multiplying by $1/\sqrt{\kappa}$, we get \[\left|\frac{1}{\sqrt{\kappa}} \hat{p}\left(\frac{x}{\kappa} - 1\right) - \frac{1}{\sqrt{x}}\right| \le \frac{\delta}{\sqrt{\kappa}}\,.\] Thus, the polynomial $q_{\kappa, \delta}(x) = \frac{1}{\sqrt{\kappa}} \hat{p}\left(\frac{x}{\kappa} - 1\right)$ has degree $O(\sqrt{\kappa} \log \frac{\kappa}{\delta})$ and satisfies $|q_{\kappa, \delta}(x) - x^{-1/2}| \le \delta/\sqrt\kappa$ for all $1 \le x \le \kappa$.

\end{proof}

\subsection{\texorpdfstring{Upper bound for $\tr(A^{-1})$}{Upper bound for tr(A-1)}}
\begin{theorem}[Polynomial approximation of $X^{-1}$]
For any $\kappa \ge 2$ and $0 < \delta < 1/2$, there exists a polynomial $p_{\kappa, \delta}(x)$ of degree $O(\sqrt{\kappa} \log \frac{\kappa}{\delta})$ such that $|p_{\kappa, \delta}(x) - x^{-1}| \le \frac{\delta}{\kappa}$ for all $1 \le x \le \kappa$.
\end{theorem}

\begin{proof}
The proof largely follows the structure of the proof for Proposition~\ref{prop:inverse_square_root_polyapprox}. We want to approximate $x^{-1}$. First, consider the function $(1+y)^{-1}$. For $|y| \le 1 - 1/\kappa < 1$, we can use the geometric series. Let $c_t = (-1)^t$, so $|c_t| = 1$ $\forall t \ge 0$:
    $$ (1+y)^{-1} = 1 - y + y^2 - y^3 + \dots = \sum_{t=0}^{\infty} (-1)^t y^t $$
\noindent
Let $S_T(y) = \sum_{t=0}^{T} (-1)^t y^t$ be the truncated sum. The error from truncation is $R_T(y) = (1+y)^{-1} - S_T(y) = \sum_{t=T+1}^{\infty} (-1)^t y^t$.
    For $|y| \le 1 - 1/\kappa$, the magnitude of the error is:
    $$ |R_T(y)| = \left|\frac{(-1)^{T+1} y^{T+1}}{1+y}\right| \le \frac{|y|^{T+1}}{1 - |y|} $$
    Since $|y| \le 1 - 1/\kappa$, $1 - |y| \ge 1/\kappa$. So,
    $$ |R_T(y)| \le \frac{|y|^{T+1}}{1/\kappa} = \kappa |y|^{T+1} \le \kappa \left(1 - \frac{1}{\kappa}\right)^{T+1} $$
We bound this error to a constrained value $\le \delta_P / 2$.
For large $\kappa$, we use the approximation:
\begin{align*}
    \left(1 - \frac{1}{\kappa}\right)^\kappa &\approx e^{-1}
\end{align*}

This implies:
\begin{align*}
    \left(1 - \frac{1}{\kappa}\right)^{T+1} &= \left(1 - \frac{1}{\kappa}\right) \left(1 - \frac{1}{\kappa}\right)^T \\
    &\approx \left(1 - \frac{1}{\kappa}\right) \left(e^{-1/\kappa}\right)^T \\
    &= \left(1 - \frac{1}{\kappa}\right) e^{-T/\kappa}
\end{align*}

Let $T = C\kappa \log(2\kappa/\delta_P)$ for some constant $C$.
We examine the term $\kappa (1 - 1/\kappa)^{T+1}$:
\begin{align*}
    \kappa \left(1 - \frac{1}{\kappa}\right)^{T+1} &\approx \kappa \left(1 - \frac{1}{\kappa}\right) e^{-T/\kappa} \\
    &\le \kappa \cdot e^{-T/\kappa} \\
    &= \kappa \cdot \exp\left(-\frac{C\kappa \log(2\kappa/\delta_P)}{\kappa}\right) \\
    &= \kappa \cdot \left(\exp\left(\log\left(\frac{2\kappa}{\delta_P}\right)\right)\right)^{-C} \\
    &= \kappa \cdot \left(\frac{2\kappa}{\delta_P}\right)^{-C} = \kappa \cdot \left(\frac{\delta_P}{2\kappa}\right)^C
\end{align*}

By choosing $C$ appropriately, we have:

\begin{align*}
    |R_T(y)| &= \left|(1+y)^{-1} - \sum_{t=0}^{T} (-1)^t y^t\right| \\
             &= \left|\frac{(-y)^{T+1}}{1+y}\right| \,, \qquad \text{where } T = O(\kappa \log(\kappa/\delta_P))
\end{align*}

Using Proposition 1 (Sachdeva \& Vishnoi, Theorem 3.3), for any positive integer $t$ and $\delta_t > 0$, there exists a polynomial $p_{t, \delta_t}(y)$ of degree $O(\sqrt{t \log(t/\delta_t)})$ such that $|p_{t, \delta_t}(y) - y^t| \le \delta_t$ for all $y \in [-1, 1]$. Let $p_{0, \delta_0}(y) = 1$.
    Let $P(y) = \sum_{t=0}^{T} (-1)^t p_{t, \delta_t}(y)$.
    The error is:
    \begin{align*} |(1+y)^{-1} - P(y)| &\le \left|(1+y)^{-1} - S_T(y)\right| + \left|S_T(y) - P(y)\right| \\ &\le \frac{\delta_P}{2} + \left|\sum_{t=0}^{T} (-1)^t (y^t - p_{t, \delta_t}(y))\right| \\ &\le \frac{\delta_P}{2} + \sum_{t=1}^{T} |(-1)^t| |y^t - p_{t, \delta_t}(y)| \\ &\le \frac{\delta_P}{2} + \sum_{t=1}^{T} \delta_t \end{align*}
We choose $\delta_t = \delta_P / (2T)$, then $|(1+y)^{-1} - P(y)| \le \delta_P / 2 + \delta_P / 2 = \delta_P$.

The degree of $P(y)$ is dominated by the degree of $p_{T, \delta_T}(y)$, which is $O(\sqrt{T \log(T/\delta_T)})$. Given $T = O(\kappa \log(\kappa/\delta_P))$ and $\delta_T = \delta_P/(2T)$, $\log(T/\delta_T) = \log(2T^2/\delta_P)$. The degree of $P(y)$ is $O(\sqrt{T \log T}) = O(\sqrt{\kappa \log(\kappa/\delta_P)} \cdot \log(\kappa \log(\kappa/\delta_P)) )$ = $O(\sqrt{\kappa} \log(\kappa/\delta_P))$.

We want to approximate $x^{-1}$ for $x \in [1, \kappa]$.
    Let $y = x/\kappa - 1$. As $x$ ranges from $1$ to $\kappa$:
    If $x = 1$, $y = 1/\kappa - 1$.
    If $x = \kappa$, $y = \kappa/\kappa - 1 = 0$.
    So, $y \in [1/\kappa - 1, 0]$. This means $|y| \le 1 - 1/\kappa$, which is the required domain for $P(y)$.
    We have $x^{-1} = (1/\kappa) \cdot (x/\kappa)^{-1} = (1/\kappa) \cdot (1 + (x/\kappa - 1))^{-1} = (1/\kappa) \cdot (1+y)^{-1}$.
    Define the approximating polynomial for $x^{-1}$ as $r_{\kappa, \delta_{\text{final}}}(x) = (1/\kappa) \cdot P(x/\kappa - 1)$.
    The error is:
    $$ |r_{\kappa, \delta_{\text{final}}}(x) - x^{-1}| = \left|\frac{1}{\kappa}P\left(\frac{x}{\kappa} - 1\right) - \frac{1}{\kappa}(1+y)^{-1}\right| = \frac{1}{\kappa} |P(y) - (1+y)^{-1}| \le \frac{1}{\kappa} \delta_P $$
    We want this final error to be $\le \delta/\kappa$. So, we set $(1/\kappa) \delta_P = \delta/\kappa$, which implies $\delta_P = \delta$.
    The degree of $r_{\kappa, \delta}(x)$ is the same as the degree of $P(y)$, which is $O(\sqrt{\kappa} \log(\kappa/\delta_P)) = O(\sqrt{\kappa} \log(\kappa/\delta))$.
\end{proof}

\section{Query lower bound via Wishart Matrices}\label{sec:wishart}
Wishart matrices are a cornerstone of multivariate statistics, and key calculations like the trace of an inverse Wishart matrix $(\tr(W^{-1}))$, are essential in many areas of statistics and machine learning. For example, they are vital for Bayesian methods involving precision matrices (the inverse of covariance matrices) and for analyzing multivariate error metrics. Therefore, understanding the fundamental limits of how efficiently we can estimate these traces is very important

We briefly present an existing theorem that establishes a lower bound on the number of matrix-vector product queries needed to estimate the trace of the inverse of a $d \times d$ Wishart matrix $W \sim \text{Wishart}(d)$. Subsequently, we extend this theorem to the trace of its inverse raised to any power $p>1/2$.

\begin{theorem}[{\cite[Lower Bound for Inverse Trace Estimation]{chewi2024query}}]\label{thm:inv_trace_lower_bd_rephrased_rephrased}
    For a $d \times d$ Wishart matrix $W \sim \text{Wishart}(d)$, where $d \ge 2$, consider any algorithm that performs $n$ matrix-vector product queries with $W$. If this algorithm outputs an estimator $\widehat{\text{tr}}$ such that, with a probability of at least $1-\delta$ (where $\delta > 0$ is dependent only on $C$), $C^{-1} \tr(W^{-1}) \le \widehat{\text{tr}} \le C\tr(W^{-1})$, then the number of queries $n$ must satisfy $n \ge \Omega(d)$.
\end{theorem}

\begin{theorem}[Lower Bound for Inverse Power Trace Estimation]\label{thm:inv_power_trace_estimation}
    For a $d \times d$ Wishart matrix $W \sim \text{Wishart}(d)$, where $d \ge 2$, consider any algorithm that performs $n$ matrix-vector product queries with $W$. If this algorithm outputs an estimator $\widehat{\text{tr}}$ such that, with a probability of at least $1-\delta$ (where $\delta > 0$ is dependent only on $C$), $C^{-1} \tr(W^{-p}) \le \widehat{\text{tr}} \le C\tr(W^{-p})$ where $p > 1/2$, then the number of queries $n$ must satisfy $n \ge \Omega(d)$.
\end{theorem}

\subsection{Preliminaries} \label{subsec:prelim_block_krylov}
We record down some useful properties of Wishart matrices that will be essential to subsequent arguments. These properties relate to how the distribution of a Wishart matrix transforms under matrix-vector product queries and provide insights into the behavior of its eigenvalues.

\begin{proposition}[{\cite[Lemma 3.4]{braverman2020gradient}}] \label{prop:wishart_posterior} 
    Let $W \sim \Wishart(d).$ Then, for any sequence of $n < d$ (possibly adaptive) queries $v_1, \dotsc, v_n$ and responses $w_1 = W v_1, \dotsc, w_n = W v_n$, there exists an orthogonal matrix $V \in \BR^{d \times d}$ and matrices $Y_1 \in \BR^{n \times n}, Y_2 \in \BR^{(d-n) \times n}$ that only depend on $v_1, \dots, v_n, w_1, \dots, w_n$, such that $V W V^\top$ has the block form
    \begin{align*}
        VWV^\top = \begin{bmatrix} Y_1 Y_1^\top & Y_1 Y_2^\top \\ Y_2 Y_1^\top & Y_2 Y_2^\top + \widetilde{W} \end{bmatrix}\,.
    \end{align*}
    Here, conditionally on $v_1, \dotsc, v_n, w_1, \dotsc, w_n$, the matrix $\widetilde{W}$ has the $\Wishart(d-n)$ distribution.
\end{proposition}
\noindent
We first record some important facts that we will use later on.
Throughout, let $K$ be an odd integer.
The following is a standard approximation-theoretic result:

\begin{proposition}[{\cite[Lemma 3.5]{braverman2020gradient}}] \label{prop:wishart_smaller_than_posterior} 
    For any matrices $Y_1 \in \BR^{n \times n}$, $Y_2 \in \BR^{(d-n) \times n}$, and any symmetric matrix $\widetilde{W} \in \BR^{(d-n) \times (d-n)}$, it holds that
    \begin{align*}
        \lambda_{\min}\Bigl(\begin{bmatrix} Y_1 Y_1^\top & Y_1 Y_2^\top \\ Y_2 Y_1^\top & Y_2 Y_2^\top + \widetilde{W} \end{bmatrix} \Bigr)
        \le \lambda_{\min}(\widetilde W)\,.
    \end{align*}
\end{proposition}

\begin{proposition}[{\cite[Extreme singular values of a Gaussian matrix]{edelman1988eigenvalues,vershynin2018high}}]\label{prop:wishart_smallest_eig}
    Let $W \sim \Wishart(d)$.
    For any $x \in [0, 1]$,
    \begin{align*}
        \Pr\bigl\{\lambda_{\min}(W) \le \frac{x}{d^2}\bigr\}
        \asymp \sqrt x\,.
    \end{align*}
    Also, there is a universal constant $C > 0$ such that
    \begin{align*}
        \Pr\{\lambda_{\max}(W) \ge C \, (1+t)\}
        &\le 2\exp(-dt)\,.
    \end{align*}
\end{proposition}

\begin{proposition}[Bound on the inverse trace] \label{prop:wishart_inverse_trace_shortened}
    Let $W \sim \Wishart(d)$.
    Then, for any $\delta > 0$, with probability at least $1-\delta$, it holds that $\tr(W^{-1}) \le C_\delta d^2$ where $C_\delta$ is a constant depending only on $\delta$.
\end{proposition}
\begin{proof}
According to~\cite[Theorem 1.2]{szarek1991condition}, for a universal constant $C > 0$, any $j=1,\dotsc,d$, and $\alpha \ge 0$:
\begin{align*}
    \Pr\Bigl\{\frac{1}{\lambda_j(W)} \ge \frac{d^2}{\alpha^2 j^2}\Bigr\} \le {(C\alpha)}^{j^2}.
\end{align*}
Let $\alpha < 1/C$. Define the event $E_\alpha = \{\text{for some } j, 1/\lambda_j(W) \ge d^2/(\alpha^2 j^2)\}$.
By the union bound:
\begin{align*}
    \Pr(E_\alpha) \le \sum_{j=1}^d {(C\alpha)}^{j^2} \lesssim \frac{1}{\sqrt{\log(1/(C\alpha))}}.
\end{align*}
We can choose $\alpha$ sufficiently small such that $\Pr(E_\alpha) \le \delta$ for any given $\delta > 0$.
On the complementary event $E_\alpha^\comp$, which occurs with probability at least $1-\delta$, we have $1/\lambda_j(W) < d^2/(\alpha^2 j^2)$ for all $j=1,\dotsc,d$.
Therefore, on $E_\alpha^\comp$:
\begin{align*}
    \tr(W^{-1}) = \sum_{j=1}^d \frac{1}{\lambda_j(W)} < \sum_{j=1}^d \frac{d^2}{\alpha^2 j^2} = \frac{d^2}{\alpha^2} \sum_{j=1}^d \frac{1}{j^2}.
\end{align*}
Since $\sum_{j=1}^d \frac{1}{j^2} \le \sum_{j=1}^\infty \frac{1}{j^2} = \frac{\uppi^2}{6}$, we have:
\begin{align*}
    \tr(W^{-1}) < \frac{\uppi^2 d^2}{6\alpha^2}.
\end{align*}
Setting $C_\delta = \frac{\uppi^2}{6\alpha^2}$ (which depends only on $\delta$ through the choice of $\alpha$), we obtain $\tr(W^{-1}) \le C_\delta d^2$ with probability at least $1-\delta$.
\end{proof}

\subsection{\texorpdfstring{Lower bound for inverse power trace estimation $\tr(W^{-q})$}{Lower bound for inverse power trace estimation tr(W-q)}}
We now present the proof of Theorem~\ref{thm:inv_power_trace_estimation}, which establishes a query lower bound for the estimation of $\text{tr}(W^{-q})$, valid for all $q > 1/2$.

\begin{proof}
\noindent
We start off by assuming, for the sake of contradiction, that $n \leq d/2$. Let $ > 0$ be the failure probability that we choose later.
\noindent
Using Proposition~\ref{prop:wishart_inverse_trace_shortened}, $\lambda_j(W^{-1}) \lesssim d^2/j^2$ with high probability for each $j=1, \dots, d$. Thus, $\lambda_j(W^{-p}) = (\lambda_j(W^{-1}))^p \lesssim (d^2/j^2)^p = d^{2p}/j^{2p}$ with high probability.
    Then, $\tr(W^{-p}) = \sum_{j=1}^d \lambda_j(W^{-p}) \lesssim \sum_{j=1}^d \frac{d^{2p}}{j^{2p}} = d^{2p} \sum_{j=1}^d j^{-2p}$.
\noindent
Constraining $p > 1/2 \implies 2p > 1 \implies \lim_{d \to \infty} \sum_{j=1}^d j^{-2p} = \zeta(2p)$, where $\zeta(s) = \sum_{n=1}^\infty n^{-s}$ is the Riemann zeta function. Therefore,
\begin{align*}
  \exists \text{ a constant } C' > 0 \text{ such that } \Pr\bigl\{\text{tr}(W^{-p}) \le C' d^{2p}\bigr\} \ge 1/2
\end{align*}
\noindent
We have a success event $E_\text{success} = \{C^{-1} \tr(W^{-1}) \le \widehat{\text{tr}} \le C\tr(W^{-1})\}$ with probability at least $1 - \delta$.
\noindent
Consider the probability that the estimator $\widehat{\text{tr}}$ is bounded:
    \begin{align*} \Pr\{\widehat{\text{tr}} \le C (C' d^{2p})\} &\ge \Pr\{\widehat{\text{tr}} \le C \tr(W^{-p}) \text{ and } \tr(W^{-p}) \le C' d^{2p}\} \\ &= \Pr\{E_{success} \text{ and } \tr(W^{-p}) \le C' d^{2p}\} \\ &\ge \Pr\{E_{success}\} + \Pr\{\tr(W^{-p}) \le C' d^{2p}\}\} - 1 \quad \\ &\ge (1-\delta) + \frac{1}{2} - 1 = \frac{1}{2} - \delta. \end{align*}
\noindent
Let $K \coloneqq C C'_p d^{2p}$. We have $\Pr\{\widehat{T}_p \le K\} \ge \frac{1}{2} - \delta$.

\noindent
Consider a failure event $E_\text{failure} = \{\widehat{\text{tr}} < C^{-1} \tr(W^{-p})\}$
where the estimate is too small.

\begin{align*} \Pr\{\widehat{T}_p < C^{-1} \tr(W^{-p})\} &\ge \Pr\{\widehat{T}_p \le K \text{ and } C^{-1} \tr(W^{-p}) > K\} \\ &= \Pr\{\widehat{T}_p \le K \text{ and } \tr(W^{-p}) > C K\}. \end{align*}
    Substituting $K = C C'_p d^{2p}$, the second condition is $\tr(W^{-p}) > C^2 C'_p d^{2p}$.
    Since $\tr(W^{-p}) = \sum \lambda_j(W^{-p}) \ge \lambda_{\max}(W^{-p}) = (\lambda_{\min}(W))^{-p}$, a sufficient condition for $\tr(W^{-p}) > C^2 C'_p d^{2p}$ is
    $$ (\lambda_{\min}(W))^{-p} > C^2 C'_p d^{2p} \implies \lambda_{\min}(W) < (C^2 C'_p d^{2p})^{-1/p} = (C^2 C'_p)^{-1/p} d^{-2} $$
\noindent
Note that $(C^2 C'_p)^{-1/p}$ is a positive constant.

Let $\mathcal{F}_n$ be the $\sigma$-algebra generated by the $n$ queries $v_1, \dots, v_n$ and responses $w_1=Wv_1, \dots, w_n=Wv_n$.
    Using Proposition~\ref{prop:wishart_posterior}, conditionally on $\mathcal{F}_n$, $W$ can be related to a posterior Wishart matrix $\widetilde{W} \sim \Wishart(d-n)$. Proposition~\ref{prop:wishart_smaller_than_posterior} states $\lambda_{\min}(W) \le \lambda_{\min}(\widetilde{W})$.
    \noindent
    \begin{align*} \therefore \Pr\{\widehat{\text{tr}} < C^{-1} \tr(W^{-p})\} &\ge \Pr\{\widehat{\text{tr}} \le K \text{ and } \lambda_{\min}(W) < (C^2 C'_p)^{-1/p} d^{-2}\} \\ &\ge \Pr\{\widehat{\text{tr}} \le K \text{ and } \lambda_{\min}(\widetilde{W}) < (C^2 C'_p)^{-1/p} d^{-2}\} \\ &= \E\left[ \mathbf{1}_{\{\widehat{\text{tr}} \le K\}} \cdot \Pr\{\lambda_{\min}(\widetilde{W}) < (C^2 C'_p)^{-1/p} d^{-2} \mid \mathcal{F}_n\} \right]. \end{align*}
    Consider a Wishart matrix $\widetilde{W} \sim \Wishart(d')$ where $d' = d-n$. Since we assumed $n \le d/2$, we have $d/2 \le d' \le d$, so $d' \asymp d$. Proposition~\ref{prop:wishart_smallest_eig} states that,
    $$ \Pr\{\lambda_{\min}(\widetilde{W}) \le x / (d')^2\} \asymp \sqrt{x} $$
    Let $x' / (d')^2 = (C^2 C'_p)^{-1/p} d^{-2}$.
    Then $x' = (C^2 C'_p)^{-1/p} (d'/d)^2$. Since $(C^2 C'_p)^{-1/p}$ is a constant and $(d'/d)^2 \asymp 1$, $x'$ is a positive constant.
    So, $\Pr\{\lambda_{\min}(\widetilde{W}) < (C^2 C'_p)^{-1/p} d^{-2} \mid \mathcal{F}_n\} = \Pr\{\lambda_{\min}(\widetilde{W}) \le x' / (d')^2 \mid \mathcal{F}_n\} \asymp \sqrt{x_0} \asymp (C^2 C'_p)^{-1/2p}$.
    Let $\epsilon$ be a positive constant (depending only on $C, p$) such that $\Pr\{\lambda_{\min}(\widetilde{W}) < (C^2 C'_p)^{-1/p} d^{-2} \mid \mathcal{F}_n\} \ge \epsilon$. This $\epsilon$ is proportional to $(C^2 C'_p)^{-1/2p}$.
    \begin{align*}\therefore \Pr\{\widehat{\text{tr}} < C^{-1} \tr(W^{-p})\} &\ge \E\left[ \mathbf{1}_{\{\widehat{\text{tr}} \le K\}} \cdot \epsilon \right] \\ &= \epsilon \Pr\{\widehat{\text{tr}} \le K\} \\ &\ge \epsilon \left(\frac{1}{2} - \delta\right). \end{align*}
    The algorithm is assumed to succeed with probability at least $1-\delta$, so the total failure probability is at most $\delta$. Thus, the probability of this specific failure mode must also be $\le \delta$:
    $$ \Pr\{\widehat{\text{tr}} < C^{-1} \tr(W^{-p})\} \le \delta. $$
    Combining these, if the algorithm works as specified and $n \le d/2$, then it must hold that
    $$ \epsilon \left(\frac{1}{2} - \delta\right) \le \delta \implies \delta \ge \frac{\epsilon}{2(1 + \epsilon)}. $$

    \noindent
     Let $\delta_0 \coloneqq \frac{\epsilon}{2(1 + \epsilon)}$. Since $\epsilon$ is a positive constant (depending only on $C$ and $p$), $\delta_0$ is also a positive constant depending only on $C$ and $p$. $\therefore$ if $n \le d/2$, the algorithm can only satisfy its success guarantee if its failure probability $\delta$ is at least $\delta_0$.
     
     However, the theorem statement allows for choosing any $\delta$ such that $0 < \delta < \delta_0$. If we choose such a $\delta$ (e.g., $\delta = \delta_0/2$), then the condition $\delta \ge \delta_0$ is violated. This is a contradiction.
     
     In conclusion, the assumption $n \le d/2$ must be false. Therefore, any algorithm meeting the specified accuracy and confidence (for a sufficiently small $\delta < \delta_0$) must use $n > d/2$ queries. This implies $n = \Omega(d)$.
\end{proof}

\section*{Acknowledgments}
The author thanks Professor
Christopher Musco for helpful conversations and mentorship. The study is done as part of the Recent Developments in Algorithm Design course at New York University.

\printbibliography

\appendix

\end{document}